\documentclass[11 pt]{article}
\usepackage[margin=1in]{geometry}
\usepackage{amsmath,amsthm,amsfonts,graphicx,float}
\usepackage{amssymb,upgreek,stmaryrd,framed}
\usepackage{wrapfig,mathrsfs,xfrac,mathtools}
\usepackage{qtree,setspace,accents,verbatim}
\usepackage[x11names]{xcolor}
\usepackage[ruled,vlined]{algorithm2e}
\usepackage{bbm,epsdice}
\usepackage[normalem]{ulem}
\usepackage[english]{babel}
\usepackage{framed}
\usepackage[shortlabels]{enumitem}

\makeatletter\let\@@citation@@=\citation\renewcommand{\citation}[1]{\@@citation@@{#1}\@for\@tempa:=#1\do{\@ifundefined{cit@\@tempa}{\global\@namedef{cit@\@tempa}{}}{}}}\makeatother
\usepackage[linktocpage=true]{hyperref}
\usepackage[capitalize]{cleveref}
\makeatletter\def\@lbibitem[#1]#2#3\par{\@ifundefined{cit@#2}{}{\@skiphyperreftrue\H@item[\ifx\Hy@raisedlink\@empty\hyper@anchorstart{cite.#2\@extra@b@citeb}\@BIBLABEL{#1}\hyper@anchorend\else\Hy@raisedlink{\hyper@anchorstart{cite.#2\@extra@b@citeb}\hyper@anchorend}\@BIBLABEL{#1}\fi\hfill]\@skiphyperreffalse}\if@filesw\begingroup\let\protect\noexpand\immediate\write\@auxout{\string\bibcite{#2}{#1}}\endgroup\fi\ignorespaces\@ifundefined{cit@#2}{}{#3}} \def\@bibitem#1#2\par{\@ifundefined{cit@#1}{}{\@skiphyperreftrue\H@item\@skiphyperreffalse\Hy@raisedlink{\hyper@anchorstart{cite.#1\@extra@b@citeb}\relax\hyper@anchorend}}\if@filesw\begingroup\let\protect\noexpand\immediate\write\@auxout{\string\bibcite{#1}{\the\value{\@listctr}}}\endgroup\fi\ignorespaces\@ifundefined{cit@#1}{}{#2}}\makeatother

\newcommand{\Ex}{\mathbb{E}}

\newcommand{\N}{\mathbb{N}}

\newcommand{\PP}{\mathbb{P}}

\newcommand{\Z}{\mathbb{Z}}

\newcommand{\cS}{\mathcal{S}}

\newcommand{\bbone}{\boldsymbol{\mathbbm1}}

\newcommand{\cev}[1]{\reflectbox{\ensuremath{\vec{\reflectbox{\ensuremath{#1}}}}}}
\DeclareMathOperator{\tail}{tail}

\newcommand{\nth}{^{\text{th}}} 

\newcommand{\eps}{\upvarepsilon}

\renewcommand{\rm}[1]{\mathrm{#1}}
\renewcommand{\tt}[1]{\texttt{#1}}

\newcommand{\ett}[1]{\emph{\texttt{#1}}}

\newcommand{\sceil}[1]{\lceil #1 \rceil}
\newcommand{\floor}[1]{\left\lfloor #1 \right\rfloor}

\DeclareMathOperator{\var}{Var}

\DeclareMathOperator{\bin}{Bin}
\DeclareMathOperator{\geom}{Geom}

\DeclareMathOperator{\rev}{rev}

\newtheorem{theorem}{Theorem}[section]
\newtheorem{lemma}[theorem]{Lemma}
\newtheorem{problem}[theorem]{Problem}

\newtheorem{proposition}[theorem]{Proposition}
\newtheorem{corollary}[theorem]{Corollary}
\newtheorem{claim}{Claim}[theorem]
\newtheorem{conjecture}[theorem]{Conjecture}

\newtheorem{innercustomtheorem}{Theorem}
\newenvironment{numtheorem}[1]
{\renewcommand\theinnercustomtheorem{#1}\innercustomtheorem}
  {\endinnercustomtheorem}

\newtheorem{innercustomlemma}{Lemma}
\newenvironment{numlemma}[1]
{\renewcommand\theinnercustomlemma{#1}\innercustomlemma}
{\endinnercustomlemma}

\theoremstyle{definition}
\newtheorem*{definition}{Definition}

\theoremstyle{remark}
\newtheorem*{remark}{Remark}
\newtheorem*{example}{Example}

\newenvironment{subproof}[1][\proofname]{
\begin{proof}[#1]}{\end{proof}}

\newcommand{\emptyword}{\epsilon}
\newcommand{\ibuf}{I}
\newcommand{\abuf}{J}
\newcommand{\turn}{z}
\DeclareMathOperator{\greedy}{gr}
\DeclareMathOperator{\NB}{NB}

\numberwithin{equation}{section}

\newcommand{\get}{=}

\title{Asymptotically half of binary words are shuffle squares}
\author{Xiaoyu He\thanks{School of Mathematics, Georgia Institute of Technology, Atlanta, GA 30332. Email: xhe399@gatech.edu.} \ \ and Logan Post\thanks{School of Mathematics, Georgia Institute of Technology, Atlanta, GA 30332. Email: lpost3@gatech.edu}}
\begin{document}
\maketitle

\begin{abstract}
    A \textit{binary shuffle square} is a binary word of even length that can be partitioned into two disjoint, identical subwords. Huang, Nam, Thaper, and the first author conjectured that as $n\rightarrow \infty$, asymptotically half of all binary words of length $2n$ are shuffle squares. We prove this conjecture in a strong form, by showing that the number of binary shuffle squares of length $2n$ is $(\frac{1}{2} - o(n^{-1/15})) 2^{2n}$. 
\end{abstract}

\section{Introduction}

Let $[k]^{*}$ be the collection of words of finite length over the finite alphabet $[k] = \{1,\ldots, k\}$, and say $s'\in [k]^*$ is a \textit{subword} of $s \in [k]^*$ if $s'$ is a (not necessarily contiguous) subsequence of $s$. A \textit{shuffle square} is a word $s\in [k]^*$ that can be partitioned into two disjoint, identical subwords. Shuffle squares were first studied by Henshall, Rampersad, and Shallit~\cite{henshall}, in the theory of formal languages. In this area, Boulteau and Vialette~\cite{binarynphard} proved that recognizing binary shuffle squares is NP-complete, building on work of Rizzi and Vialette~\cite{unshufflelargek} and Buss and Soltys~\cite{unshuffle7buss} on larger alphabets.

As shuffle squares must have even length, it is convenient to say that a word of length $2n$ has \textit{semi-length} $n$. Let $\cS_k(n)$ be the set of shuffle squares of semi-length $n$ over the alphabet $[k]$. Henshall, Rampersad, and Shallit asked for the asymptotics of $|\cS_k(n)|$, and gave conjectural asymptotic formulas for $n$ fixed and $k\rightarrow\infty$. These formulas were verified by the first author, Huang, Nam, and Thaper~\cite{HHNT}, who also posed the following surprising conjecture in the opposite regime, where $k=2$ is fixed and $n\rightarrow \infty$, supported by numerical evidence from the OEIS~\cite{oeis}.

\begin{conjecture}[\cite{HHNT}]
Asymptotically half of all even-length binary words are shuffle squares, that is, $|\cS_2(n)| = (\frac{1}{2}-o(1)) 2^{2n}$.
\end{conjecture}

Note that shuffle squares must have an even count of each letter, so this conjecture states that almost all words satisfying this obvious parity constraint are actually shuffle squares. Towards this, they proved using a greedy algorithm that $|\cS_2(n)| \ge \binom{2n}{n} = \Theta(n^{-1/2}2^{2n})$. Our main theorem resolves this conjecture in the affirmative with an explicit polynomial error term.

\begin{theorem}\label{thm:square}
We have $|\cS_2(n)|=(\frac{1}{2}-o(n^{-1/15})) 2^{2n}$.
\end{theorem}

The exponent $-1/15$ is not optimized, but within a modest factor of the best possible using our arguments. It would be interesting to decide if the error term is actually polynomial or exponential. 

Closely related to the shuffle square problem is the ``twins'' problem introduced by Axenovich, Person, and Puzynina~\cite{Axenovich}. For a word $s\in [k]^*$, let $\rm{LT}(s)$ be the maximum semi-length of a subword of $s$ that is a shuffle square, so $s\in \cS_k(n)$ if and only if $\rm{LT}(s)=n$. Equivalently, $\rm{LT}(s)$ is the largest $m$ such that $s$ contains ``twins'' of length $m$, i.e. disjoint identical subwords of length $m$. Axenovich, Person, and Puzynina proved that for any $s\in \{0,1\}^{2n}$, we have $\rm{LT}(s)=(1-o(1))n$. More progress has been made on twins for various alphabets~\cite{basu,bukhzhou}, and the twins problem is related to regularity lemmas for words, the longest common subsequence problem \cite{bukhhogenson,bukhma,ChvatalSankoff,He-Li}, and codes for correcting deletion errors \cite{Guruswami-He-Li}.

While previous work dealt with the worst-case behavior of $\rm{LT}(s)$, \cref{thm:square} implies the best possible bound for its average-case behavior. 
\begin{corollary}\label{corollary:lt} For almost all $s\in\{0,1\}^{n}$, $\rm{LT}(s)\geq \sceil{\frac{n}{2}}-1$.
\end{corollary}

\begin{proof} Let $s\in \{0,1\}^{2n}$ be uniformly random among words with an even number of $0$'s and $1$'s. If $\rm{LT}(s)=n$, then by deleting one or two bits, we immediately have $\rm{LT}(s[1,2n-1])=n-1$ and $\rm{LT}(s[1,2n-2])\geq n-2$. Thus by \cref{thm:square},
\[
\PP[\rm{LT}(s[1,2n-2])\geq n-2]\geq \PP[\rm{LT}(s[1,2n-1])=n-1]\geq \PP[s\in \cS_2(n)]=1-o(1).
\]
Moreover, $s[1,2n-1]$ is uniformly random in $\{0,1\}^{2n-1}$ and $s[1,2n-2]$ is uniformly random in $\{0,1\}^{2n-2}$, implying the corollary.
\end{proof}

In \cite{HHNT}, it was conjectured that almost every binary word can be made into a shuffle square by deleting at most three bits - \cref{corollary:lt} shows that it actually suffices to delete two (which is best possible).

\par{\bf Organization.} In Section~\ref{section:preliminaries}, we introduce the notion of a \textit{buffer} for a binary word, thereby reformulating the problem as a discrete stochastic process on subsets of $\{0,1\}^*$. We define and analyze the greedy algorithm using buffers and obtain asymptotic results on its behavior. In Section~\ref{section:buffersetanalysis}, we analyze the evolution of the set of all buffers to show that a random word $s$ can ``absorb" small defects in an almost-perfect partition. In Section~\ref{section:boostedgreedy}, we give an improved algorithm which produces a shuffle square with constant probability. This boosted greedy algorithm stores local information about the word and occasionally backtracks for an improvement. Together with the buffer set analysis, this proves~\cref{thm:square}. In Section~\ref{section:conclusion}, we discuss the shuffle square problem on larger alphabets, and we provide a boosted greedy algorithm for $k$-ary words which gives a new lower bound for $|\cS_k(n)|$.

\section{Preliminaries}\label{section:preliminaries}

For a word $s\in [k]^n$ we use $s(i)$ to denote the $i\nth$ letter of $s$. We use $u\circ v$ or simply $uv$ to denote concatenation of words, and $\emptyword$ to denote the empty word. The word $\rev(s)$ is obtained by reversing the order of indices, so $\rev(s)(i)=s(n-i+1)$. For a subset $A\subseteq[n]$, we write $s(A)$ for the subword of $s$ indexed by $A$. For a subinterval $[i,j]$ of $[n]$, we write $s[i,j] \coloneqq s([i,j])$. We also define $\tail(s) \coloneqq s[2, n]$ for the operation of removing the first letter.

The negative binomial distribution appears many times in our calculations, so we denote $\rm{NB}(k)\coloneqq \rm{NegBin}(k,1/2)$. By convention, $X\sim \rm{NB}(k)$ is supported on $\N=\{0,1,\ldots\}$, and we have $\Ex[X]=k$ and $\var[X]=2k$.

In this paper we limit our focus to binary words on alphabet $\{0,1\}$. For a word $s\in \{0,1\}^n$, we say that $w\in\{0,1\}^*$ is a \textit{buffer} for $s$ if there is a partition $(A_1,A_2)$ of the index set $[n]$ such that $s(A_1)=s(A_2)\circ w$. Let $B_t(s)$ be the set of buffers of the prefix $s[1,t]$. Observe that $s\in\{0,1\}^n$ is a shuffle square if and only if $\emptyword \in B_n(s)$. When the word $s$ is clear from context, we write $B_t$ for $B_t(s)$.

\begin{proposition}\label{prop:buffersetdefn}Fix some word $s\in \{0,1\}^n$. Then $B_0=\{\emptyword\}$, and for all $1\leq t\leq n$,
\begin{equation*}
B_t=\big\{w\circ s(t)\mid  w\in B_{t-1}\big\}\cup \big\{\tail(w)\mid w\in B_{t-1},\ w(1)=s(t)\big\}.
\end{equation*}
\end{proposition}
\begin{proof} Observe that $w'\in B_t$ if and only if there is a partition $s(A_1),s(A_2)$ of $s[1,t-1]$ such that either $s(A_1\cup \{t\})=s(A_2)\circ w'$ or else $s(A_1)=s(A_2\cup \{t\})\circ w'$. If $w'=\emptyword$, then the two cases are identical up to exchanging $A_1$ and $A_2$. Otherwise, in the first case $w'=w\circ s(t)$ with $w\in B_t$. In the second case we have $w'=\tail(w)$ with $w=s(t)\circ w'\in B_t$. 
\end{proof}

Thus, if $s$ is uniformly random in $\{0,1\}^n$ then by reading off one bit at a time, $B_t\subseteq \{0,1\}^*$ evolves by a discrete stochastic process in $t$, where $|B_t|$ is nondecreasing. 

\subsection{Buffer threads and the Greedy Algorithm}\label{section:greedyalgorithm}

We say that a sequence $(w_t)_{t=0}^n$ is a \textit{buffer thread} for a word $s\in \{0,1\}^{n}$ if $w_0=\emptyword$ and for all $t\geq 1$, either $w_t=w_{t-1}\circ s(t)$ or both $w_t=\tail(w_{t-1})$ and $s(t)=w_{t-1}(1)$. Observe that for all $t$, $w_t\in B_t(s)$. Conversely, for every word $w\in\{0,1\}^*$, we have $w\in B_t$ if and only if there is some buffer thread $(w_t)_t$ for $s$ with $w_t=w$. Moreover, the buffer thread $(w_t)_{t=0}^n$ exactly encodes a partition of $[n]$ for which $w_n$ is the buffer. Indeed, if we let $A_1$ be the set of indices $t$ for which $w_t = w_{t-1} \circ s(t)$, and $A_2 = [n]\setminus A$, we obtain that $s(A_1) = s(A_2) \circ w_n$.

To show that $s$ is close to a shuffle square, it suffices to construct a buffer thread such that $w_n$ is short. The following greedy algorithm from \cite{HHNT} obtains a particularly simple buffer thread.

\begin{definition} The greedy algorithm $\greedy:\{0,1\}^*\to \{0,1\}^*$ computes one buffer for an input $s$. We define the function recursively: $\greedy(\emptyword)=\emptyword$, and for all $s\in \{0,1\}^t$, if $\greedy(s[1,t-1])=w$, then
\begin{equation}\label{eq:greedydefn}
\greedy(s)=\begin{cases}
w\circ s(t)&\text{ if $s(t)\neq w(1)$}\\
\tail(w)&\text{ if $s(t)= w(1)$}.
\end{cases}
\end{equation}
Essentially, this algorithm always chooses to shorten the buffer when possible. This gives a buffer thread $(\greedy_t(s))_{t\geq 0}$ where $\greedy_t(s)=\greedy(s[1,t])$.
\end{definition}

\begin{example}For the word $s=100011100$, we obtain the following sequence:
\begin{center}
\begin{tabular}{c|ccccccccc}
$s(t)$&1&0&0&0&1&1&1&0&0\\
$\greedy_t(s)$&1& 10 &100 &1000& 000& 0001& 00011& 0011 &011
\end{tabular}
\end{center}
\end{example}
Define the set $\Sigma_2=\{0^a1^b,1^a0^b\mid a,b\geq 0\}$, the binary words with at most two runs. We observe that $\greedy_t(s)\in \Sigma_2$ for any time $t$ and any binary word $s$. Indeed, if $\greedy_{t-1}(s)=0^a1^b$ where $a>0$, then we will have $\greedy_t(s)\in \{0^{a-1}1^b,0^a1^{b+1}\}$, depending on the value of $s(t)$. Additionally, note the parity constraint $|\greedy_t(s)|\equiv t\pmod 2$ for all $t\geq 0$. The following proposition essentially rephrases Lemma 5.2 in \cite{HHNT}.

\begin{proposition}\label{prop:greedysrw} If $s\in \{0,1\}^\omega$ is uniformly random, then $(|\greedy_t(s)|)_{t=0}^\infty$ is a simple random walk on $\N$ with a reflective barrier at $0$. As a result, for all $t\geq 0$, we have $\PP[\greedy_{2t}(s)=\emptyword]={2t\choose t}2^{-2t}$. If $k\neq 0$ and $k\equiv t\pmod 2$, then $\PP[|\greedy_t(s)|=k]=2{t\choose \frac{t+k}{2}}2^{-t}$.
\end{proposition}

Roughly, this follows from the fact at each time $t$ we have $|\greedy_{t+1}(s)|=|\greedy_{t}(s)|\pm 1$, with each step equally likely. We now perform a more careful analysis of the probability that $\greedy_t(s)$ takes any particular value.

\begin{definition}For a word $w\in \{0,1\}^*$, let $q_t(w)=\PP[\greedy_t(s)=w]$ where $s\in\{0,1\}^{\geq t}$ is uniformly random. For $k\in\N$, let $q_t(k)=\PP[|\greedy_t(s)|=w]=\sum_{|w|=k}q_t(w)$. 
\end{definition}
\noindent Note that that Proposition~\ref{prop:greedysrw} gives the exact value of $q_t(k)$, and from this we can recover the result of the first author, Huang, Nam, and Thaper in \cite{HHNT}:
\[|\cS_2(n)|=2^{2n}\PP[s\in \cS_2(n)]=2^{2n}\PP[\emptyword\in B_{2n}]\geq 2^{2n}q_t(\emptyword)={2n\choose n}.\]
Recall that we always have $\greedy_t(s)\in \Sigma_2$, so if $v\notin \Sigma_2$ then $q_t(v)=0$. For a fixed $k\geq 1$, we have $|\Sigma_2\cap \{0,1\}^k|=|\{0^a1^{k-a},1^a0^{k-a}\mid 1\le a\le k\}|=2k$. It is natural to predict that among the $2k$ words of length $k$, the greedy algorithm visits each word with roughly the same frequency. We prove that this heuristic holds when \textit{both $k$ and $t$} are sufficiently large. In Section~\ref{section:conclusion}, we discuss the unusual behavior of $q_t(w)$ for small words; for example $q_t(01)$ and $q_t(00)$ have different asymptotic constants as $t\to\infty$.

\begin{lemma}\label{lem:greedyuniform} For all $w\in \Sigma_2$ with $|w|=k\geq 3$, if $t+k$ is even then
\begin{equation}\label{eq:qt}
\frac{1}{2(k-2)}q_t(k-2)
\geq q_t(w)\geq \frac{1}{2(k+2)}q_t(k+2), 
\end{equation}
so in particular, if $t\ge k^2$, then $q_t(w)=\Theta(k^{-1}t^{-1/2})$ and as $t^2/k,k\to\infty$, $q_t(w)\sim \frac 1{2k}q_t(k)$.
\end{lemma}
\begin{proof}
We show a monotonicity property: when $t$ is even,
\[
q_t(\emptyword)\geq q_t(10)\geq q_t(11)\geq q_t(1000)\geq q_t(1100) \geq q_t(1110)\geq\ldots,
\]
and when $t$ is odd we have $q_t(1)\geq q_t(100)\geq q_t(110) \geq \ldots$.
That is, for all $t,a,b\geq 0$, we have $q_t(1^a)\geq q_t(10^{a+1})$ and $q_t(1^{a+1}0^{b+1})\geq q_t(1^{a+2}0^b)$. We proceed by induction on $t$. When $t=0$, this holds trivially because $q_t(\emptyword)=1$ and for all $w\neq \emptyword$, $q_t(w)=0$. We require some recurrence relations for $q_t$; towards this let $s\in \{0,1\}^\omega$ be uniformly random so that $q_t(w)=\PP[\greedy_t(s)=w]$. Suppose the claim holds for $t\geq 0$ and consider the word $1^a$ for $a\geq 0$. If $a\neq 1$ then conditioning on the value of $s(t+1)$, we obtain the following recurrence relations. By symmetry between $01^a$ and $10^a$ and by induction,
\begin{align*}
q_{t+1}(1^a)&=\frac 12q_t(01^a)+\frac 12q_t(1^{a+1})\\
&\geq \frac 12q_t(10^a)+\frac 12q_t(1^20^{a+1})=q_{t+1}(10^{a+1}),
\end{align*}
In the case $a=1$, we have $q_{t+1}(1)=\frac 12q_t(01)+\frac 12(q_t(11)+q_t(\emptyword))\geq q_{t+1}(10^{2})$, applying the calculation above with the extra term $q_t(\emptyword)$. Now, let $a,b\geq 0$, and consider the word $1^{a+1}0^{b+1}$. If $b\neq 0$, then conditioning on $s(t+1)$,
\begin{align*}
q_{t+1}(1^{a+1}0^{b+1})&=\frac 12 q_t(1^{a+1}0^{b})+\frac 12 q_t(1^{a+2}0^{b+1})\\
&\geq 
\frac 12q_t(1^{a+2}0^{b-1})+\frac 12 q_t(1^{a+3}0^{b})=
q_{t+1}(1^{a+2}0^{b}),
\end{align*}
by induction, term-wise. If $b=0$, then
\begin{align*}
q_{t+1}(1^{a+1}0)&=\frac 12 q_t(1^{a+1})+\frac 12 q_t(1^{a+2}0)\\
&\geq \frac 12q_t(01^{a+2})+\frac 12 q_t(1^{a+3})=
q_{t+1}(1^{a+2}),
\end{align*}
also by induction, term-wise. This completes the proof by induction, so the monotonicity holds. Also by symmetry $q_t(1^a0^b)=q_t(0^a1^b)$, and thus for all $w,v\in \Sigma_2$ with $k=|w|<|v|$, we have $q_t(w) \geq q_t(v)$. By averaging over $v$ of length $k+2$,
\[
q_t(w)\geq \frac 1{2(k+2)}\sum_{v \in\Sigma_2\cap\{0,1\}^{k+2}}\!\!q_t(v)=\frac 1{2(k+2)}q_t(k+2).
\]
By averaging over $k-2$, we obtain the analogous upper bound when $k\geq 3$. This proves \eqref{eq:qt}. Since 
\[
q_t(k)=2{t\choose \frac{t+k}{2}}2^{-t} = \Theta(t^{-1/2}),
\]
if $t\ge k^2$, the second conclusion of the lemma follows.
\end{proof}

\section{Buffer Set Analysis}\label{section:buffersetanalysis}

By itself, the greedy algorithm succeeds to obtain $\emptyword \in B_n(s)$ with probability $\Theta(n^{-1/2})$. In Section~\ref{section:boostedgreedy}, we give a boosted version of the greedy algorithm which succeeds with constant probability. Both of these algorithms yield individual buffer threads which attempt to find a partition of the input word $s$ into identical subsequences. We think of the buffer set $(B_t(s))_t$ as a non-deterministic algorithm which simultaneously produces all possible partitions of $s$. In this section, we study the dynamics of $B_t$ assuming the existence of a suitably strong boosted greedy algorithm which typically produces buffers of bounded size, giving the following.

\begin{lemma}\label{lem:boundedbuffer} For a uniformly random word $s\in \{0,1\}^{\omega}$, let $Y_t=\min\{|w|\mid w\in \Sigma_2\cap B_t(s)\}$. Then for all $n\in \N$, and $\delta<1/3$, $\sum_{t\leq n}\Ex[Y_{t}^\delta]=O(n)$.
\end{lemma}

In this section, it will be useful to consider a generalization of the greedy algorithm \textit{initialized} at some word $v\in\{0,1\}^*$. We let $\greedy(\emptyword;v)\coloneqq v$ and let $\greedy(s;v)$ satisfy the greedy recurrence in \eqref{eq:greedydefn}. Thus, for example, if $s\in \{0,1\}^\omega$ is uniformly random then $(|\greedy_t(s;v)|)_{t\geq 0}$ is a simple random walk initialized at $|v|$. For $w\in \{0,1\}^*$, we also let $q_t(w;v)=\PP[\greedy_t(s;v)=w]$. Indeed, since this is only a difference in initialization we have the conditional probability,
\[
q_t(w;v)=\PP[\greedy_{t+i}(s)=w\mid \greedy_i(s)=v],
\]
as long as $q_i(v)\neq 0$. As an aside, we remark that because $\greedy_t$ is a recurrent random walk, we have that if $|w|,|v|\leq k$ then $q_t(w;v)=(1+O(k/\sqrt t))q_t(w)$. In this section, let $\Sigma_2^{(01)}=\{0^a1^b\mid a\geq 0,b\geq 1\}$ and $\Sigma_2^{(10)}=\{1^a0^b\mid a\geq 0,b\geq 1\}$, so $\Sigma_2=\Sigma_2^{(01)}\cup \Sigma_2^{(10)}\cup \{\emptyword\}$. Define the probability
\[
\delta_t(v,w)=\PP[v\in B_t(s)\text{ and }w\notin B_t(s)].
\]
Our arguments will only involve the quantities $\delta_t(1w1,w)$ and $\delta_t(0w0,w)$ for $w\in \Sigma_2$. We will ultimately show that the $\delta_t$'s are small. This means that whenever, say, $1^{2k}\in B_t$, then it is also likely that $1^{2k-2},1^{2k-4},\ldots,\emptyword\in B_t$. This will imply that if $Y_t$ is bounded, then $Y_t\leq 2$ with high probability.

The quantity $\delta_t(1w1,w)$ is natural to study because at time $t$, where $s(t) = 1$ say, the buffer thread may have the option of making the greedy choice $B \mapsto \tail(B) = w$, but instead choose to make the non-greedy choice $B\mapsto B \circ 1 = 1w1$.

\begin{lemma}\label{lem:pt_decomposition} Let $s\in\{0,1\}^\omega$ be uniformly random. Let $\cev q_t(w;v)=q_t(\rev(w);\rev(v))$.
Then for all $v\in\{0,1\}^*$,
\begin{equation}\label{eq:pt_as_qt}
\PP[v\in B_n(s)]=\cev q_n(\emptyword;v)+\frac 12\sum_{t=0}^{n-1}\sum_{w\in \{0,1\}^*}\Big(\delta_t(1w1,w)\cev q_{n-t-1}(w1;v)+\delta_t(0w0,w)\cev q_{n-t-1}(w0;v)\Big).
\end{equation}
In particular, we have
\begin{equation}\label{eq:pt_as_qt_beta}
\PP[\emptyword\in B_n(s)]=q_n(\emptyword)+\sum_{t=0}^{n-1}\sum_{w1\in \Sigma_2^{(01)}}\delta_t(1w1,w) q_{n-t-1}(\rev(w1)).
\end{equation}
\end{lemma}
Roughly speaking, \eqref{eq:pt_as_qt_beta} captures the fact that any successful buffer thread $(w_j)_j$ has a critical first time $n-t-1$ at word $w_{n-t-1}=1w$ where it diverges from the greedy algorithm.
\begin{proof}
We prove \eqref{eq:pt_as_qt} by induction on $n$. By construction, $\PP[v\in B_0]=\bbone_{v=\emptyword}=q_0(\emptyword;\rev(v))$. Observe that for any word $v1$ ending in $1$, we have
\begin{align*}
\PP[v1\in B_{n+1}]&=\PP[s(n+1)=1]\PP[(v\in B_n)\cup (1v1\in B_n)]+\PP[s(n+1)=0]\PP[0v1\in B_n]\\
&=\frac 12\Big(\PP[v\in B_n]+\delta_n(1v1,v)+\PP[0v1\in B_n]\Big).
\end{align*}
Observe that $q_n$ satisfies a similar recurrence: for each initializiation of the form $1v'$, we can condition on the first bit of $s$ to get
\begin{align*}
q_n(w;1v')&=\PP[s(1)=1]\PP[\greedy(s[2,n],v')=w]+\PP[s(1)=0]\PP[\greedy(s[2,n],1v'0)=w]\\
&=\frac 12(q_{n-1}(w;v')+q_{n-1}(w;1v'0)).
\end{align*}
Reversing the words, we obtain
\[
\cev q_n(w;v1)=\frac 12(\cev q_{n-1}(w;v)+\cev q_{n-1}(w;0v1)).
\]
Denote the RHS of~\eqref{eq:pt_as_qt} by $Q_n(v)$. Since the above is a linear recurrence relation, $Q_n$ obeys the same equation, plus the added terms at time $n+1$:
\[
Q_{n+1}(v1)=\frac 12(Q_n(v)+Q_n(0v1))+\frac 12\sum_{w\in \{0,1\}^*}\Big(\delta_n(1w1,w)\cev q_{0}(w1;v1)+\delta_n(0w0,w)\cev q_{0}(w0;v1)\Big)
\]
Recall that $\cev q_0(w;v)=\bbone_{w=v}$, so the summation evaluates to $\frac 12\delta_n(1v1,v)$. Thus by induction, $\PP[v1\in B_{n+1}]=Q_{n+1}(v1)$. By symmetry, $\PP[v0\in B_{n+1}]=Q_{n+1}(v0)$, showing \eqref{eq:pt_as_qt} holds for all nonempty words $v$ at time $n+1$.

Finally, we check that $\PP[\emptyword\in B_{n+1}]=Q_{n+1}(\emptyword)$. Indeed we have $\PP[\emptyword\in B_{n+1}]=\frac 12(\PP[1\in B_n]+\PP[0\in B_n]\big)$ and for all words $w$, $q_{n+1}(w;\emptyword)=\frac 12(q_n(w;0)+q_n(w;1))$. Then by linearity, we have $Q_{n+1}(\emptyword)=\frac 12(Q_n(1)+Q_n(0))=\PP[\emptyword\in B_{t+1}]$, since all of the terms involving $q_0(w1;\emptyword)$ are zero. This completes the proof of~\eqref{eq:pt_as_qt}.

For~\eqref{eq:pt_as_qt_beta}, note that if $w\notin \Sigma_2$ then $q_t(\rev(w))=q_t(w)=0$ for all $t$. Thus all $\cev q_t(w1;\emptyword)$ terms vanish except those where $w1$ (resp.\ $w0$) has at most two runs.
With this, we obtain
\[
\PP[\emptyword \in B_n]=\cev q_n(\emptyword;\emptyword)+\frac 12\sum_{t=0}^{n-1}\bigg(\sum_{w1\in \Sigma_2^{(01)}}\delta_t(1w1,w)\cev q_{n-t-1}(w1;\emptyword)+\sum_{w0\in \Sigma_2^{(10)}}\delta_t(0w0,w)\cev q_{n-t-1}(w0;\emptyword)\Bigg).
\]
By symmetry between $0$ and $1$, the two inner sums are equal, so the equation simplifies to~\eqref{eq:pt_as_qt_beta}.
\end{proof}
\noindent Motivated by Lemma~\ref{lem:pt_decomposition}, define
\[
\delta^*_t(k)\coloneqq \sum_{\substack{w1\in \Sigma_2^{(01)}\\|w1|\leq k
}}\delta_t(1w1,w).
\]
Let us motivate the next step. Using Lemma~\ref{lem:greedyuniform}, we roughly approximate $q_{n-t-1}(w)\asymp 1/|w|\sqrt n$ for $w\in \Sigma_2$. Then Lemma~\ref{lem:pt_decomposition} roughly says that $\PP[\emptyword\in B_n]\asymp \sum_{t\leq n} \delta^*_t(k)/k\sqrt n$ for some appropriate choice of $k$. As the probability is at most $1$, we conclude that $\sum_t\delta_t^*(k)$ is not too large. The next lemma leverages this bound by comparing $\delta_t^*$ to the property of ``being more than two steps away from a shuffle square."

\begin{lemma}\label{lem:delta_event} Let $E_t(k)$ be the event that there exists $w\in \Sigma_2\cap B_t$ with $|w|\leq k$, and also $\emptyword,0,1,01,10\notin B_t$. Then
\[
\PP[E_t(k)]\leq 16\sum_{t'=t}^{t+2k-1}\delta_{t'}^*(2k)
\]
\end{lemma}
\begin{proof}
Let $D_t(k)$ be the event that for some $j\in \{0,1\}$ and some $wj\in \Sigma_2$ with $|wj|\leq k$, we have $jwj\in B_t$ but $w\notin B_t$. By a union bound, $\PP[D_t(k)]\leq 2\delta_t^*(k)$. We will show that conditioned on $E_t(k)$, with constant probability $\bigcup_{t'=t}^{t+2k-1}D_{t'}(2k)$ holds.

Let $(a,b)$ be the lexicographically minimal pair of integers for which $a+b\le k$ and $0^a1^b \in B_t$, and $(a,b) = (\infty, \infty)$ if no such pair exists. Similarly pick $(x,y)$ lexicographically minimal such that $1^x0^y\in B_t$ and $x+y\leq k$. Note that assuming $E_t(k)$, one of $a,x$ is at most $k-1$. We proceed by conditioning on the values of $a,b,x,y$. If $a=0$, then $1^b\in B_t$ and $1^{b-2}\notin B_t$ so $D_t(k)$ holds. If $a=1$, then note that $b\geq 2$ since $0,01\notin B_t$. Then with probability $1/2$ the next bit of $s$ is $s(t+1)=1$, so $1^b\in B_{t+1}$ but $1^{b-2}\notin B_{t+1}$, and thus $D_{t+1}(k)$ holds. By symmetry between $(a,b)$ and $(x,y)$, if $x=0$ or $x=1$ then $D_t(k)$ or $D_{t+1}(k)$ occurs with probability at least $1/2$. By construction, $y=0$ implies $a=0$ and similarly $b=0$ implies $x=0$.

Suppose that $x\geq a\geq 2$, so then $b,y\geq 1$. With probability at least $1/8$, we have $s(t+1)=0$ and also the next $2a-2$ bits of $s$ contain at least $a$ zeros (and hence at most $a-2\leq x-2$ ones). Conditioned on this event, let $t'$ be the position of the $a\nth$ next zero in $s$; that is, $t'\in [t+a+1,t+2a-2]$ such that $s[t+1,t']$ contains $a$ zeros and $c$ ones, $1\leq c\leq a-2$, and $s(t')=0$. Observe that the greedy algorithm yields $\greedy(s[t+1,t'];0^a1^b)=1^{b+c}$, so $1^{b+c}\in B_{t'}$. Conditioned on $t'$, we now prove that either $\emptyword,1 \notin B_{t'}$, implying $D_{t'}(k)$, or else $D_{t''}(2k)$ holds for some $t''<t'$. Since $t'<t+2k$, this is the desired conclusion.

 Suppose that $\emptyword\in B_{t'}$ or $1\in B_{t'}$, so there is a buffer thread $(w_r)_{r=t}^{t'}$ with $w_t\in B_t$ and $w_{t'}\in \{\emptyword,1\}$. First, suppose that $w_r\in \Sigma_2$ for all $r\in [t,t']$. By the minimality of $(a,b)$ and $(x,y)$, we have that either $w_t=1^i0^j$ with $i\geq x$ or $w_t=0^i1^j$ with $i\geq a$.

First, suppose $w_t=1^i0^j$, $i\geq x$. Since $s[t+1,t']$ contains $c$ ones and $c < a \le x \le i$, it is not possible for the prefix $1^i$ to be matched. Specifically, there are at most $c$ `down-steps' $r$ such that $|w_r|-|w_{r-1}|=-1$, so we have $|w_{t'}|\geq x+a-c>1$, a contradiction. 

Next, suppose that $w_t=0^i1^j$ with $i\geq a$. Since $s[t+1,t'-1]$ contains $a-1<i$ zeros, it is not possible for the prefix $0^i$ to be matched; thus $w_r(1)=0$ for all $r\in [t,t'-1]$. Including $t'$, there are at most $a$ down-steps. If $i>a$, then we also have $w_{t'}(1)=0$, contradicting that $w_{t'}\in \{\emptyword,1\}$. Suppose that $i=a$, so by minimality $j\geq b$. Then because there are at most $a$ down-steps, we have $|w_{t'}|\geq b+c\geq 2$, a contradiction. We conclude that there is no such thread $(w_r)_r$ with $w_r\in \Sigma_2$ for all $r$. 

Thus, for each buffer thread $(w_r)_r$ with $w_{t'}\in \{\emptyword, 1\}$, there is some maximal $t \le t''<t'$ such that $w_{t''}\notin \Sigma_2$. Among all such buffer threads, pick one minimizing $t''$. We know $w_{t''+1}\in \Sigma_2$. If we suppose $w_{t''+1}\in \Sigma_2^{(01)}$, then $w_{t''+1}=0^i1^j$ and $w_{t''}=10^i1^j$ for some $i,j\geq 1$, and $s(t''+1)=1$. Moreover, we must have $v=0^i1^{j-1}\notin B_{t''}$ or else we could have chosen to set $w_{t''}\coloneqq v\in \Sigma_2$, contradicting the minimality of $t''$. Thus we have both $v\notin B_{t''}$ and $1v1\in B_{t''}$. Since $(|w_r|)_r$ is Lipschitz and $|w_{t'}|\leq 1$, surely $|v1|\leq t'-t''+1\leq 2k$, so the event $D_{t''}(2k)$ holds. The case $w_{t''+1}\in \Sigma_2^{(10)}$ follows by symmetry. 

By the law of total probability over $a,b,x,y$, we conclude
\[
\PP\Big[\bigcup_{t'=t}^{t+2k-1}D_t(2k)\mid E_t(k)\Big]\geq \Ex_{a,b,x,y\mid E_t(k)}\Big[\PP\Big[\bigcup_{t'=t}^{t+2k-1}D_t(2k)\mid a,b,x,y\Big]\Big]\geq 1/8.
\]
Using the identity $\PP[E]\leq \PP[D\mid E]^{-1}\PP[D]$, we obtain the main inequality.
\end{proof}

\noindent We are now ready to restate and prove the main theorem.

\begin{numtheorem}{\ref{thm:square}}[Equivalent form]
If $s\in \{0,1\}^{2n}$ is uniformly random, then $\PP[\emptyword \in B_{2n}(s)]= \frac 12-o(n^{-1/15})$.
\end{numtheorem}
\begin{proof}
With foresight, let $k=\floor{n^{3/14}}=o(\sqrt n)$. In order to apply \cref{lem:delta_event}, our first goal is to show that $\delta^*_t(2k)$ is small on average. 

By Lemma~\ref{lem:greedyuniform}, for all $t\leq 2n+2k$ and $w\in \Sigma_2$ with $|w|\leq 2k$, if $t+|w|$ is even, then $q_{4n-t-1}(\rev(w1)) =\Omega(k^{-1}n^{-1/2})$. Note that if $t+|w|$ is odd then $\delta_t(1w1,w)=0$. By Lemma~\ref{lem:pt_decomposition}, artificially extending $s$ to length $4n$ by sampling $2n$ more i.i.d uniform bits, we have
\[
1\geq \PP[\emptyword\in B_{4n}]\geq \sum_{t=0}^{2n+2k}\sum_{\substack{w1\in \Sigma_2^{(01)}\\|w1|\leq 2k}}\delta_t(1w1,w)q_{4n-t-1}(\rev(w1))= \Omega(k^{-1}n^{-1/2})\sum_{t=0}^{2n+2k}\delta_t^*(2k).
\]
Here, we applied \eqref{eq:pt_as_qt_beta} and used the fact that all terms are nonnegative to drop terms in the middle inequality. Thus, 
\[
\sum_{t=0}^{2n+2k}\delta^*_t(2k)=O(k\sqrt n).
\]
Let $A_t$ be the event that $\emptyword,0,1,01,10\notin B_t$. Let $Y_t=\min\{|w|\mid w\in \Sigma_2\cap B_t\}$. As in Lemma~\ref{lem:delta_event}, let $E_t(k)$ be the event that $Y_t\leq k$ and $A_t$ holds. By Lemma~\ref{lem:delta_event},
\[\PP[A_t]=\PP[Y_t>k]+\PP[E_t(k)]\leq \PP[Y_t>k]+16\sum_{r=t}^{t+2k-1}\delta_r^*(2k).
\]

Note that for $0<\gamma<1/3$, Markov's inequality implies $\PP[Y_t>k]\leq \Ex[Y_t^{1/3-\gamma}]/k^{1/3-\gamma}$. We will take $\gamma < 0.001$. By Lemma~\ref{lem:boundedbuffer}, this expectation is bounded in Ces\`aro mean. Thus, summing over $t$ we have
\[
\sum_{t=0}^{2n}\PP[A_t]\leq \sum_{t=0}^{2n}\frac{\Ex[Y_t^{1/3-\gamma}]}{k^{1/3-\gamma}}+\sum_{t=0}^{2n+2k}2k\delta_t^*(2k)\leq O(nk^{-1/3+\gamma})+O(k^2n^{1/2})\leq O(n^{13/14+\gamma})=o(n^{14/15}),
\]
where we used our assumptions that $k\sim n^{3/14}$ and $\gamma<0.001$. By averaging, there is some odd (deterministic) time $t^*\leq 2n$ such that $\PP[A_{t^*}]+\PP[A_{2n-t^*}]\leq o(n^{-1/15})$. 

Now, we use this special time $t^*$ to show that with probability close to $\frac12$, a uniform random word $s$ of length $2n$ is a shuffle square. Break $s$ into two intervals $s[1,t^*]$ and $s[t^*+1, 2n]$, and condition on the event of probability $\frac 1 2$ that $s$ has an even number of $0$'s and $1$'s. Observe that $s[1,t^*]$ and $s[t^*+1, 2n]$ are each independent of this event, and so are still uniform random on the margin. Applying the definition of $A_{t^*}$ to $s[1,t^*]$ and the definition of $A_{2n-t^*}$ to $\rev(s[t^*+1, 2n])$ and because $t^*$ is odd, we obtain that with conditional probability at least $1-o(n^{-1/15})$ there is some $j_1\in \{0,1\}\cap B_{t^*}(s)$ and also some $j_2\in \{0,1\}\cap B_{n-t^*}(\rm{rev}(s))$. Thus we have a partition $(A_1,A_2)$ of $[t^*]$ and a partition $(A_1',A_2')$ of $[t^*+1,2n]$ such that $s(A_1)=s(A_2)\circ j_1$ and $s(A_1')=j_2\circ s(A_2')$. Since $s$ has an even number of each letter, $j_1=j_2$, and therefore the partition $(A_1\cup A_2',A_2\cup A_1')$ yields a partition of $s$ into identical sub-words. This succeeds with probability at least $\frac 12-o(n^{-1/15})$, as desired.
\end{proof}

\section{A Locally Boosted Greedy Algorithm} \label{section:boostedgreedy}
It remains to show Lemma~\ref{lem:boundedbuffer}, that the length of the minimum buffer element has a bounded moment. Recall that the greedy algorithm gives a buffer thread whose length behaves like an unbiased simple random walk (with a reflective barrier at $0$). We give an improved greedy algorithm which occasionally makes non-greedy choices with foresight obtain a walk with sufficient negative drift such that (a certain moment of) the resulting buffer thread is bounded in size. The algorithm proceeds in cycles; the following lemma outlines how to iterate the algorithm for one cycle. One can view the following lemma as a statement about the recurrence of a stochastic dynamical system: we describe a boosted greedy algorithm which tracks its first return to a state when the buffer thread is a word of the form $1^k$ or $0^k$.

\begin{lemma}\label{lem:algorithm_step}Let $s\in \{0,1\}^{\omega}$ be uniformly random, let $k\geq 0$, and suppose $1^k\in B_{t_0}(s)$ at some initial time $t_0 \ge 0$. Then there are random variables $X^*,T^*$ such that one of $0^{X^*},1^{X^*}$ is in $B_{t_0+T^*}(s)$, satisfying:
\begin{align}
\Ex[X^*-k]&=-2+O((7/8)^k) \label{eq:algorithmstep1}\\
\Ex[(X^*-k)^2]&\leq 3k+O(1) \label{eq:algorithmstep2}\\
\PP[T^*\geq (4+\eps)k]&\leq \exp(-\Omega(\eps^2k/(1+\eps))) \label{eq:algorithmstep3}\\
\PP[|X^*-k|\geq \eps k]&\leq \exp(-\Omega(\eps^2k/(1+\eps))), \label{eq:algorithmstep4}
\end{align}
for any $\eps > 0$.
\end{lemma}

\begin{remark} It is possible to use this lemma to implement a linear-time algorithm to find an explicit shuffle square partition of a random shuffle square $s\in \{0,1\}^{2n}$, with success probability $1-o(1)$. Algorithms~\ref{alg:indicatorphase},~\ref{alg:turnoverphase}, and~\ref{alg:activationphase} implement the three phases of this lemma in pseudocode. By iterating the lemma, we obtain an increasing sequence of times $(T_m)_m$ together with sequence of constant buffers $({j_m}^{X_m})_m$, where $j_m \in \{0,1\}$. In Lemma~\ref{lem:boosted_algorithm}, we show that $X_m$ remains small on average. Repeating the algorithm until $T_{m+1}\approx n-\log n$, we then use brute force to optimally partition the remaining suffix of $s$ on $O(n)$ time. We omit the details. This average-case result contrasts with the theorem of Boulteau and Vialette~\cite{binarynphard} that the decision problem is NP-complete in general.
\end{remark}

Before we prove Lemma~\ref{lem:algorithm_step}, we introduce several definitions. The idea is to iteratively construct a `quasi-buffer' thread $(W_t)_t$ of words in $\{0,1,\rm i\}^*$ in three phases such that the quasi-buffer tracks a set of possible buffers at each phase. We refer to the `i's as \textit{indicators}. Roughly speaking, an indicator i indicates the location of a special $1$ for which we are delaying the decision of whether or not to match it until later bits of $s$ are revealed.

\begin{example}
    Suppose $111\in B_{t_0}$ and $s[t_0+1,t_0+10]=0110111010$. Phase 1, the indicator phase, reads bits in from $s$ until three $1$'s are seen. In this case the bits are $01101$. The quasi-buffer after this phase is $0\rm i0$, signifying that both $1010$ and $00$ are buffers.
    
    Phase 2, the turnover phase, reads in the next run of $1$'s along with the single following $0$. In this case it is $110$, so we record that $\rm i0$ is a quasi-buffer now, signifying that both $101$ and $011$ are buffers.

    Phase 3, the activation phase, reads until it sees as many $0$'s as are in the quasi-buffer (in this case $1$). In this case it sees $10$. Using this information, it identifies that $101$ is the more advantageous result from the previous phase, matching to find that $1$ is in the final buffer. The greedy algorithm would have produced the longer buffer $111$.
\end{example}

For a word $w\in \{0,1,\rm i\}$, let $\overline w=w(\{j\mid w(j)\neq \rm i\})$ be the word obtained from $w$ by deleting all indicators. Let $\#\rm i(w)=\{j\mid w(j)=\rm i\}$ be the number of indicators in $w$. An i\textit{-partition} of $w$ is a tuple $(u_0,u_1,\ldots, u_r)$ of words in $\{0,1,\rm i\}^*$ with $0\leq r\leq \#\rm i(w)$ such that $w=u_0\rm iu_1\cdots \rm i u_r$. Observe that a word $w$ has $2^{\#\rm i(w)}$ distinct i-partitions. For example, $(w)$ is a trivial i-partition of $w$. For this section, let $s\in \{0,1\}^\omega$ be uniformly random and $B_t=B_t(s)$. We proceed algorithmically in three phases; first is the indicator phase.

\begin{definition}Define a set $\ibuf_t\subseteq \{0,1,\rm i\}^*$ of \textit{indicator-buffers} at time $t$. We say that $w\in \ibuf_t$ if for every i-partition $w=u_0\rm iu_1\cdots \rm i u_r$, we have $1^r\overline{u_0}1\overline{u_1}\cdots 1\overline{u_r}\in B_t$. 
\end{definition}

\noindent An equivalent recursive condition for $w\in \ibuf_t$ is that either $w\in B_t$ (in the case $w$ has no indicators) or for any i-bipartition $w=u_0\rm iu_1$, we have both $u_0u_1\in \ibuf_t$ and $1u_01u_1\in \ibuf_t$ (corresponding to a non-greedy decision at i). Also, note that $\ibuf_t$ is closed under deletion of indicators, and thus if $w\in \ibuf_t$ then $\overline w\in B_t$.

The first of three phases we call the \textit{indicator phase}. If the initial buffer is $1^k$, then the indicator phase reads in bits of $s$ until exactly $k$ $1$'s are reached, and places indicators at the locations of those $1$'s immediately followed by $0$'s. 

\begin{lemma}[Indicator phase]\label{lem:indicatorphase} Let $k\geq 1$. If $1^k\in B_{t_0}$, then there exist random variables $X\sim \NB(k)$ and $M\sim \max(\bin(k-1,1/2)-1,0)$, and $T_1=t_0+k+X$ such that, if $X=0$, then $\emptyword\in B_{T_1}$, and if $X>0$, then there exist $X_{0},\ldots, X_{M}>0$ such that $\sum_{j=0}^MX_j=X$ and we have
\[
0^{X_{0}}\rm i0^{X_{1}}\rm i\cdots \rm i0^{X_{M}}\in \ibuf_t.
\]
\end{lemma}

\begin{algorithm}
\label{alg:indicatorphase}
\DontPrintSemicolon
\KwOut{A time $T^*$ and word $W_{T^*}$ equaling either $0^{X^*}$ or $1^{X^*}$.}
$W_0 \get 1^k$\;
\For{$t\get 1$ \KwTo $\infty$}{
\uIf{$s(t) = 0$}{
    $W_t \get W_{t-1} \circ 0$\;
}
\ElseIf{$s(t) = 1$}{
    $W_t \get \tt{tail}(W_{t-1}) \circ \text{i}$\;
    \If{$W_t(1) \in \{0,\emph{i}\}$}{
        $T_1\get t$\;
        Obtain $W_t'$ from $W_t$ by deleting each `i' not followed by $0$.\;
        Obtain $W_t''$ from $W_t'$ by deleting the leftmost `i', if any exists.\;
        \eIf{$W_t''=\emptyword$}{
        \Return{$(\emptyword,T_1)$}
        }{
        \Return{\ett{Turnover\_Phase}$(W_t'',s,T_1)$}}
    }
}
}
\caption{\texttt{Indicator\_Phase}$(s,k)$}
\end{algorithm}

\begin{proof} For simplicity, we assume $t_0=0$. Let $W_0=1^k$. Choose the unique prefix of $s$ with the structure $s[1,T_1]=0^{G_0}10^{G_1}\cdots 0^{G_{k-1}}1$, so $T_1$ is the position of the $k\nth$ 1. Note that $G_0,\ldots, G_{k-1}\sim \geom(1/2)$ are mutually independent and $X\coloneqq \sum_iG_i\sim \NB(k)$ counts the total number of $0$'s. Let $T_1=k+X$, so these have the desired distributions by construction.

For all $t\leq T_1$, construct $W_t$ recursively. If $s(t)=0$ then $W_t=W_{t-1}\circ 0$, and if $s(t)=1$ then $W_t=\rm{tail}(W_{t-1})\circ \rm i$. Observe that for each $t$, there are some $\ell,g$ such that $s[1,t]=0^{G_0}10^{G_1}\cdots 0^{G_{\ell-1}}10^{g}$, and then we have $W_t=1^{k-\ell}0^{G_0}\rm i0^{G_1}\cdots 0^{G_{\ell-1}}\rm i0^{g}$. Now we prove by induction on $t$ that for all $t\leq T_1$, $W_t\in\ibuf_t$.

By definition $W_0=1^k\in I_0$. Observe that for all $1\leq t\leq T_1$, we have $W_{t-1}(1)=1$. Suppose that $s(t)=0$ and $W_{t-1}\in \ibuf_{t-1}$. We want to show that $W_t=W_{t-1}\circ 0\in \ibuf_t$, so let $(u_0,\ldots,u_r0)$ be an i-partition of $W_t$. By induction $1^r\overline{u_0}1\cdots 1\overline{u_r}\in B_{t-1}$, so by~\cref{prop:buffersetdefn}, $1^r\overline{u_0}1\cdots 1\overline{u_r}0\in B_{t}$, which proves that $W_t\in \ibuf_t$.

Suppose that $s(t)=1$, and let $(u_0,\ldots, u_r)$ be an i-partition of $W_t$. There are two cases. If $u_r= \emptyword$, then $(1u_0,\ldots, u_{r-1})$ is an i-partition of $W_{t-1}$, so $1^r\overline{u_0}1\cdots 1\overline{u_{r-1}}\in B_{t-1}$, and therefore $1^r\overline{u_0}1\cdots 1\overline {u_{r-1}}1\overline\emptyword\in B_t$ as desired. If $u_r\neq \emptyword$ then $u_r=u_r'\rm i$ and $(1u_0,\ldots, u_r')$ is an i-partition of $W_{t-1}$. Thus $1^{r+1}\overline{u_0}1\cdots 1\overline {u_r'}\in B_{t-1}$, so $1^{r}\overline{u_0}1\cdots 1\overline {u_r}\in B_t$, as desired. We conclude
\[
W_{T_1}=0^{G_0}\rm i0^{G_1}\rm i\cdots 0^{G_{k-1}}\rm i\in \ibuf_{T_1}.
\]
If $X=0$, then $\overline{W_{T_1}}=\emptyword\in B_{T_1}$. If not, we modify $W_{T_1}$ to simplify future analysis. Let $W_{T_1}'$ be obtained from $W_{T_1}$ by deleting each indicator which is not directly followed by a zero, i.e.\ the $j\nth$ indicator is deleted if $G_j=0$. Let $W''_{T_1}$ be obtained from $W'_{T_1}$ by deleting the leftmost remaining indicator, if any exists. Observe that we can write
\[
W''_{T_1}=0^{X_{0}}\rm i0^{X_{1}}\rm i\cdots \rm i0^{X_{M}},
\]
with all $X_i$ positive and summing to $X$. We also have $W''_{T_1} \in \ibuf_{T_1}$ because $\ibuf_{T_1}$ is closed under deleting indicators. Finally, we observe that $M\sim \max(\bin(k-1,1/2)-1,0)$, since each of the first $k-1$ indicators survives the map $W_{T_1}\mapsto W'_{T_1}$ independently with probability $1/2$, and then one more is deleted in the map $W_{T_1}'\mapsto W_{T_1}''$. If $X=0$, then for consistency we set $M=0$.
\end{proof}

\noindent Next, we pass the output from the indicator phase to the turnover phase. This phase is the simplest: if the next run of $s$ is a run of $1$'s, we record this in the (quasi-)buffer. If the next run is a run of $0$'s, turnover fails and we end this cycle of the boosted greedy algorithm prematurely, acting exactly the same as the unmodified greedy algorithm.

\begin{definition} At time $t$, for $\turn\geq 0$, define a set $\abuf_{t,\turn}\subseteq \{0,1,\rm i\}^*$ of \textit{activation-buffers} as follows. We say that $w\in \abuf_{t,\turn}$ if for every i-partition $w=u_0\rm iu_1\cdots \rm i u_r$ such that $r\leq \turn$, we have $\overline{u_0}1\overline{u_1}\cdots 1\overline{u_r}1^{\turn-r}\in B_t$.
\end{definition}

An equivalent recursive condition for $w\in \abuf_{t,z}$ is that $w1^z\in B_t$ or else for any i-bipartition $w=u_0\rm iu_1$ we have $u_0u_1\in \abuf_{t,z}$, and if $z>0$ then $u_01u_0\in \abuf_{t,z-1}$. We remark that the indicators have a different interpretation in $J_{t,\turn}$ than in $I_t$. In $I_t$, the indicators could be substituted for $1$'s with additional 1's inserted on at the \textit{front} of $w$, whereas in $J_{t,\turn}$, we have a cache of $z$ ``extra" 1's at the \textit{end} of $w$, and we can insert these $1$'s in the place of up to $z$ indicators.

\begin{lemma}[Turnover Phase]\label{lem:turnoverphase} Let $t_1\in\N$, and suppose $0W\in\ibuf_{t_1}$ for $W\in \{0,1,\rm i\}^*$. Then there exists $Z\sim \geom(1/2)$ and $T_2=t_1+Z+1$, such that $W\in J_{T_2,Z}$. In particular, if $Z=0$ then $\overline W\in B_{T_1+1}$.
\end{lemma}

\begin{algorithm}
\label{alg:turnoverphase}
\DontPrintSemicolon
$Z \get 0$\;
\eIf{$s(t_1+1)=0$}{
    $W_{t_1+1}\get \overline{\tt{tail}(W_{t_1}'')}$\;
    \Return{$(W_{t_1+1},t_1+1)$}
}{
$Z\get 1$\;
\For{$t\get t_1+2$ \KwTo $\infty$}{
\uIf{$s(t) = 1$}{
    $Z \get Z+ 1$\;
}
\ElseIf{$s(t)=0$}{
    $T_2\get t$\;
    $W_{T_2}\get \tt{tail}(W_{t_1}'')$\;
    \Return{\ett{Activation\_Phase}$(W_{T_2},s,T_2,Z)$}
}}}
\caption{\tt{Turnover\_Phase}$(W_{t_1}'',s,T_1)$}
\end{algorithm}

\begin{proof} Let $T_2$ be the position of the next zero in $s$, so $s[t_1+1,T_2]=1^Z0$ with $Z\sim \rm{Geom}(1/2)$. Now, we prove directly that $W\in \abuf_{T_2,Z}$. Let $(u_0,\ldots, u_r)$ be an i-partition of $W$ such that $r\leq z$. Because $0W\in I_{t_1}$, we have $1^r\overline{0u_0} 1\cdots 1\overline{u_r}\in B_{t_1}$. Then by deleting the first $r$ 1's and appending the next $Z-r$, we have $\overline{0u_0} 1\cdots 1\overline{u_r} 1^{Z-r}\in B_{t_1+Z}$. Then deleting the first $0$, we have $\overline{u_0}1\cdots 1\overline{u_r} 1^{Z-r}\in B_{T_2}$.
\end{proof}

\noindent Lastly, we pass the output of the turnover phase into the activation phase, which reads in bits until it matches every $0$ in this output.

\begin{lemma}[Activation phase]\label{lem:activationphase} Let $t_2\in\N$, $z\geq 1$. Suppose $W= 0^{x_0-1}\rm i0^{x_1}\rm i\cdots \rm i0^{x_m}\in \abuf_{t_2,z}$ for $m\geq 0$ and $x_0,\ldots, x_m>0$, and let $x=\sum_ix_i$. Then there exists $Y\sim \NB(x-1)$, $C_3\sim \min(\bin(m,1/2),z)$ and $T_3=t_2+Y+x-1$ such that $1^{Y-2C_3+z}\in B_{T_3}$.
\end{lemma}

\begin{algorithm}
\label{alg:activationphase}
\DontPrintSemicolon
$C_3\get 0$\;
\For{$t\get t_2+1$ \KwTo $\infty$}{
\uIf{$s(t) = 0$}{
    Delete any leading `i's from $W_{t-1}$.\\
    $W_t\get \tt{tail}(W_{t-1})$\;
}
\uElseIf{$s(t) = 1\emph{\textbf{ and }}W(1) = \emph{\text{i}}\emph{\textbf{ and }}C_3<z$}{
    $C_3 \get C_3 + 1$\;
    $W_t\get \tt{tail}(W_{t-1})$\;
}
\ElseIf{$s(t) = 1$}{
    $W_t\get W_{t-1}\circ 1$\;
}

\If{$W_t(1)=1$}{
    \Return{$ W_t\circ 1^{z-C_3}$}
}}
\caption{\tt{Activation\_Phase}$(W_{t_2},s,t_2,z)$}
\end{algorithm}

In this phase the variable $C_3$ gives the key ``boost" to the greedy algorithm. Observe that there is exactly one case (iii) below where the buffer length decreases when the greedy buffer would have increased. In this case, we are using an indicator to make an under-the-hood substitution using the previous two phases such that ultimately $X^*$ decreases by $2$ on average.

\begin{proof}
Let $T_3$ be the time of the $(x-1)\nth$ next zero in $s$, so $s[t_2+1,T_3]=1^{G_0}0\cdots 1^{G_{x-2}}0$, where $G_0,\ldots, G_{x-2}\sim \geom(1/2)$ are are mutually independent and $Y\coloneqq \sum_iG_i\sim \NB(x-1)$ is the number of $1$'s in this interval. Let $W_{t_2}=W$ and for all $t\in [t_2+1,\ldots, T_3]$, construct $W_t$ iteratively according to the following casework. Let $C_k(t)$ be the number of times $t'\leq t$ that case $k$ out of 4 occurs, initialized by $C_k(t_2)=0$.
\begin{enumerate}[(i)]
\setlength{\itemsep}{0pt}
\item If $s(t)=0$ and $W_{t-1}(1)=0$, let $W_t=\rm{tail}(W_{t-1})$.
\item If $s(t)=0$ and $W_{t-1}(1)=\rm i$, $W_{t-1}(2)=0$, let $W_t=\rm{tail}^2(W_{t-1})$.
\item If $s(t)=1$ and $W_{t-1}(1)=\rm i$, $W_{t-1}(2)=0$, and $C_3(t-1)<z$, let $W_t=\rm{tail}(W_{t-1})$.
\item If $s(t)=1$ and $W_{t-1}(1)=0$ or $C_3(t-1)=z$, let $W_t=W_{t-1}\circ 1$.
\end{enumerate}
The critical cases (ii) and (iii) exploit the flexibility of the artificial i character: if the corresponding input bit is a $0$, then we delete the $\rm i$ and match the next $0$, while if the corresponding bit is a $1$, then we treat $\rm i$ as a $1$. The variable $C_3(t)$ tracks that we have a total budget of $z$ i's that can be used as $1$'s. 

By this iteration, each step (i), (ii) deletes a zero and step (iv) adds a $1$, so $W_{T_3}=1^{C_4}=1^{Y-C_3}$
\begin{claim}\label{claim:activationinduction} For all $t\in [t_2,T_3]$, $W_t\in \abuf_{t,z-C_3(t)}$. 
\end{claim}
\begin{subproof}[Proof of Claim] We apply induction on $t$. The base case $W_{t_2}\in \abuf_{t_2,z}$ is given. Assuming $W_{t-1} \in J_{t-1,z-C_3(t-1)}$, we show that $W_t \in J_{t,z-C_3(t)}$ under each of the four cases.

In case (i), let $(u_0,\ldots, u_r)$ be an i-partition of $W_t$ with $r\leq z-C_3(t)$. Since $C_3(t)=C_3(t-1)$, by assumption $\overline{0u_0}1\cdots \overline{u_r}1^{z-C_3(t)-r}\in B_{t-1}$, so then $\overline{u_0}1\cdots \overline{u_r}1^{z-C_3(t)-r}\in B_{t}$, proving that $W_t\in \abuf_{t,z-C_3(t)}$. 

In case (ii), by closure under indicator deletion, $\rm{tail}(W_{t-1})\in \abuf_{t-1,z-C_3(t)}$. Now identically to case (i), we have $\rm{tail}^2(W_{t-1})=W_t\in \abuf_{t,z-C_3(t)}$.

In case (iii), let $(u_0,\ldots, u_r)$ be an i-partition of $W_t$ with $r\leq z-C_3(t)=z-C_3(t-1)-1$. Since $W_{t-1}(1)=\rm i$, $(\emptyword,u_0,\ldots, u_r)$ is an i-partition of $W_{t-1}$, so $1\overline{u_0}1\cdots\overline{u_r}1^{z-C_3(t-1)-(r+1)}\in B_{t-1}$. Because $s(t)=1$, we have $\overline{u_0}1\cdots\overline{u_r}1^{z-C_3(t)-r}\in B_t$, which means $W_t\in \abuf_{t,z-C_3(t)}$.

In case (iv), let $(u_0,\ldots, u_r1)$ be an i-partition of $W_t$. Then $C_3(t)=C_3(t-1)$ and by induction
$\overline{u_0} 1\cdots\overline{u_r}1^{z-C_3(t)-r}\in B_{t-1}$. Since $s(t)=1$, we have $\overline{u_0} 1\cdots\overline{u_r1}1^{z-C_3(t)-r}\in B_t$, so $W_t\in \abuf_{t,z-C_3(t)}$, completing the proof by induction.
\end{subproof}
By Claim~\ref{claim:activationinduction}, we have $1^{Y-C_3}\in J_{T_3,z-C_3}$, where $C_3 \coloneqq C_3(T_3)$, and thus $1^{Y-2C_3+z}\in B_{T_3}$ as desired. It suffices to show that $C_3\sim \min(\bin(m,1/2),z)$. Suppose we have some time $t\in [t_2+1,T_3]$ such that $W_{t-1}(1)=\rm i$. At each of these times, with probability $1/2$ we have $s(t)=1$, in which case step (iii) will occur assuming $C_3 < z$. Note that regardless of $s(t)$, the leading indicator of $W_{t-1}$ will be deleted at time $t$. Since $\#\rm i(W_{t_2})=m$, there will be exactly $m$ such times $t\in [t_2+1,T_3]$, implying the desired distribution of $C_3$.
\end{proof}

\noindent Before proving the full \cref{lem:algorithm_step}, we give a quick Chernoff bound for negative binomials.

\begin{proposition}\label{prop:nbchernoff}
If $X\sim \NB(k)$, then for $\eps>0$,
\[
\PP[X\geq (1+\eps)k],\PP[X\leq (1-\eps)k]\leq e^{-\eps^2k/2(2+\eps)}.
\]
\end{proposition}
\begin{proof}We use the identity that for all $m\in\N$,
$
\PP[X\geq m]=\PP_{Y\sim \bin(m+k-1,1/2)}[Y\geq m].
$
First, we round up so that $\eps k$ is an integer. Letting $m=(1+\eps)k$, and $Y\sim \bin((2+\eps)k,1/2)$, standard Chernoff bounds for the binomial distribution give 
\[
\PP[X\geq (1+\eps)k]\leq \PP[Y\geq \Ex[Y]+\eps k/2]\leq e^{-2(\eps k/2)^2/(2+\eps)k}=e^{-\eps^2k/2(2+\eps)}.
\]
The lower tail is similar.
\end{proof}

We are now ready to prove the main lemma.

\begin{proof}[Proof of ~\cref{lem:algorithm_step}]
Suppose $1^k\in B_{t_0}$; for simplicity assume $t_0=0$. We break into three cases depending on whether this iteration of the boosted greedy algorithm ends in the indicator phase, the turnover phase, or the activation phase. 

\textbf{Indicator phase.} \cref{lem:indicatorphase} gives random variables $X,M,T_1$, and if $X>0$ also $X_0,\ldots, X_M$ with the properties specified. If $X=0$, we say this iteration ends in the indicator phase and we let $X^*=X$ and $T^*=T_1$. 

\textbf{Turnover phase.} If $X>0$, then \cref{lem:turnoverphase} gives random variables $Z,T_2$, where $T_2=T_1+Z+1$, with the properties specified. If $Z=0$, then we say this iteration ends in the turnover phase. By the given properties, $0^{X-1}\in B_{T_2}$, so we let $X^*=X-1$ and $T^*=T_2$. 

\textbf{Activation phase.} If $Z>0$, then we say this iteration ends in the activation phase. \cref{lem:activationphase} gives random variables $Y,T_3$, and $C_3$ such that $1^{Y-2C_3+Z}\in B_{T_3}$, and then we let $X^*=Y-2C_3+Z$ and $T^*=T_3$. 

In all three cases, we have $0^{X^*}$ or $1^{X^*}$ in $B_{T^*}$ as desired. To unify the analysis of the three cases, we extend the domains of $Z,Y,C_3$ such that $(Z\mid X=0)\sim \geom(1/2)$, $(Y\mid X=0)=0$, and $(C_3\mid Z=0\vee M=0)=0$. With these extensions, $Y,C_3$ are compound random variables with conditional distribution
\[(Y\mid X)\sim \NB(\max(X-1,0))\quad \text{ and }\quad (C_3\mid M,Z)\sim \min(\bin(M,1/2),Z).\]
Let $C'=Z-C_3$, so $Y-2C_3+Z=Y-Z+2C'$ and we decompose $X^*$ in the following way, using that $Y\bbone_{X>0}=Y$ and $C'\bbone_{X,Z>0}=C'$,
\begin{align}
X^*&=0\bbone_{X=0}+(X-1)\bbone_{X>0,Z=0}+(Y-Z+2C')\bbone_{X,Z>0}\notag\\
&=(X-1)\bbone_{Z=0}+\bbone_{X=Z=0}+(Y-Z)\bbone_{Z>0}+Z\bbone_{X=0}+2C'\notag\\
&=\underbrace{(X-1)\bbone_{Z=0}+(Y-Z)\bbone_{Z>0}}_{ X^*_1}+\underbrace{\max(Z,1)\bbone_{X=0}+2C'}_{ X^*_2}\label{eq:X*expression}
\end{align}

We prove that $X^*_1$ has the desired statistics and $X^*_2$ is a negligible perturbation. First, consider $X^*_1$. Since $X,Z$ are independent, we immediately have $\Ex[(X-1)\bbone_{Z=0}]=\frac12 (k-1)$, and $\Ex[(X-1-k)^2\bbone_{Z=0}]=\frac 12(2k+1)$. Next we control the second term in $X_1^*$; both technical claims below are proven in Section~\ref{section:numericalclaims}.
\begin{claim}\label{claim:Ybounds} (Proved in~\cref{section:numericalclaims}.) With $X,Y,Z$ as defined above, $\Ex[(Y-Z)\bbone_{Z>0}]=\frac 12(k-3+2^{-k})$ and $\Ex[(Y-Z-k)^2\bbone_{Z>0}]<4k+9$.
\end{claim}
Together, we have
\begin{align*}
\Ex[X^*_1-k]&=\frac 12(k-1)+\frac 12(k-3+2^{-k})-k=-2+2^{-k-1},\quad\text{and}\\
\Ex[(X^*_1-k)^2]&=\Ex[(X^*_1-k)^2\bbone_{Z=0}]+\Ex[(X^*_1-k)^2\bbone_{Z>0}]<\frac 12(2k+1)+\frac 12(4k+9)=3k+5.
\end{align*}
Next, we control $X^*_2$. Because $\PP[X=0]
=2^{-k}$, it is simple to check that $\Ex[\max(Z,1)\bbone_{X=0}]=\frac 32\cdot 2^{-k}$ and $\Ex[\max(Z,1)^2\bbone_{X=0}]=\frac 72\cdot 2^{-k}$, which are very small. The next claim bounds the $2C'$ term.
\begin{claim}\label{claim:C'bounds}(Proved in~\cref{section:numericalclaims}.) With $M,Z,C'$ as defined above, $\Ex[C']\leq 2(7/8)^{k'}$ and $\Ex[(C')^2]\leq 6(7/8)^{k'}$. 
\end{claim}
We conclude that $\Ex[X^*_2],\Ex[(X^*_2)^2]=O((7/8)^k)$. Therefore,
\[
\Ex[X^*-k]=\Ex[X^*_1-k]+\Ex[X^*_2]=-2+O((7/8)^k),
\]
and similarly by the Minkowski inequality,
\[
\Ex[(X^*-k)^2]\leq \Big(\sqrt{\Ex[(X^*_1-k)^2]}+\sqrt{\Ex[(X^*_2)^2]}\Big)^2=3k+5+o(1)
\]
This proves~\eqref{eq:algorithmstep2}. To obtain tail bounds~\eqref{eq:algorithmstep3} and~\eqref{eq:algorithmstep4}, first note that 
\begin{equation}\label{eq:T*expection}
T^*\leq T_3= (k+X)+(Z+1)+(X-1+Y),
\end{equation}
and thus $\Ex[T^*]\leq 4k$. By considering cases, we have
\[\PP[T^* \geq (4+5\eps)k]\leq \PP[X\geq (1+\eps)k]+\PP[Y\geq (1+2\eps)k\mid X<(1+\eps)k]+\PP[Z\geq \eps k]\]
By the definition of $Z$, we have $\PP[Z\geq \eps k]=2^{-\eps k}$. For the other two terms, recall that $X\sim \rm{NB}(k)$, and the conditional distribution $\big(Y\mid X< (1+\eps)k\big)$ is stochastically dominated by $\rm{NB}((1+\eps)k)$. Both terms can be bounded by Proposition \ref{prop:nbchernoff}, so 
\[
\PP[T^*\geq (4+5\eps)k]\leq 3e^{-\Omega(\eps^2k/(1+\eps))},
\]
which gives~\eqref{eq:algorithmstep3}. We obtain tail bounds on $X^*$ similarly. Using~\eqref{eq:X*expression}, we concentrate $X^*$ by bounding $X,Y$, and $Z$, noting that $|2C'-Z|\leq |Z|$:
\begin{multline*}
\PP[|X^*-k|\geq 3\eps k]\leq \PP[Z\geq \eps k]+\PP[|X-k-1|\geq \eps k]+
\PP\big[|Y-k|\geq 2\eps k\ \big|\  |X-k-1|<\eps k\big].
\end{multline*}
Verifying that each of these is bounded by $e^{-\Omega(\eps^2k/(1+\eps))}$, we conclude~\eqref{eq:algorithmstep4}.
\end{proof}

Now, we iterate this lemma as an algorithm to read an input word and output a short buffer. Let $s\in\{0,1\}^\omega$ be uniformly random and $(X_m)_{m\geq 0}$ and $(T_m)_{m\geq 0}$ be the sequence generated by repeated application of Lemma~\ref{lem:algorithm_step}. Specifically, let $X_0=T_0=0$, and then for each $m\in\N$, conditioned on $X_m=k$ there exists $X_{m+1}$ and $T_m^*$ satisfying Equations~\eqref{eq:algorithmstep1}-\eqref{eq:algorithmstep4} such that $0^{X_{m+1}}$ or $1^{X_{m+1}}$ is in $B_{T_m+T_m^*}$. Letting $T_{m+1}=T_m+T_m^*$, this process repeats indefinitely. Our next lemma bounds the Ces\`aro mean of some moment of $X_m$. 

\begin{lemma}\label{lem:boosted_algorithm} Let $(X_m)_{m\geq 0}$ and $(T_m)_{m\geq 0}$ be defined above. Then for all $\delta<1/3$ there exists $c>0$ such that $\sum_{m\leq n}\Ex[X_m^{1+\delta}]\leq cn$ for all $n\in \N$.
\end{lemma}

As an aside, note that $(X_m)_m$ is a time-homogeneous, irreducible, aperiodic Markov chain on $\N$. We expect that more careful analysis could improve this to a bound uniform in $m$, though this would not affect our final result.

\begin{proof} For this argument, we need a second order approximation, applying the results of Lemma~\ref{lem:algorithm_step}.
\begin{claim}\label{claim:secondorderapproxBoostedsequence} (Proved in~\cref{section:numericalclaims}.) Let $\Delta$ have the conditional distribution of $(X_m-X_{m-1}\mid X_{m-1}=k)$, with $(X_m)_m$ as defined above. For $p>0$, we have
\[
\Ex[(k+\Delta)^p-k^p]=k^{p-1}\Big(\frac p2(3p-7)+O(k^{-1/2})\Big).
\]
\end{claim}
\noindent Applying Claim~\ref{claim:secondorderapproxBoostedsequence} with $p=2+\delta$ where $\delta\in (0,1/3)$, we have
\[
\Ex[X_m^p-X_{m-1}^p\mid X_{m-1}=k]=k^{1+\delta}\Big(\frac {2+\delta}2(3\delta-1)+O(k^{-1/2})\Big)\leq c_1-c_2k^{1+\delta},
\]
for some constants $c_1,c_2>0$, depending on $\delta$. Taking expectation over $X_{m-1}$, we have
\[
\Ex[X_m^{2+\delta}]\leq \Ex[X_{m_1}^{2+\delta}]+c_1-c_2\Ex[X_{m-1}^{1+\delta}].
\]
Summing over all $m\leq n$ gives the bound $0\leq \Ex[X_{n+1}^{2+\delta}]\leq \sum_{m\leq n}(c_1-c_2\Ex[X_{m}^{1+\delta}])$, and therefore $\sum_{m\leq n}\Ex[X_m^{1+\delta}]\leq c_3n$, where $c_3=c_1/c_2$.
\end{proof}

Using this moment bound, we restate and prove the final lemma required for the buffer analysis.

\begin{numlemma}{\ref{lem:boundedbuffer}}
For a uniformly random word $s\in \{0,1\}^{\omega}$, let $Y_t=\min\{|w|\mid w\in \Sigma_2\cap B_t(s)\}$. Then for all $n\in \N$, and $\delta<1/3$, $\sum_{t\leq n}\Ex[Y_{t}^\delta]=O(n)$.
\end{numlemma}
\begin{proof}
Let $(X_m)_{m=0}^\infty$ and $(T_m)_m$ be defined as above, by repeated application of Lemma~\ref{lem:algorithm_step}. By definition, for all $m\in \N$, we have $Y_{T_m}\leq X_m$. Note that $(T_m)_m$ is strictly increasing, so we can define the random variable $M(t)=\max\{m\in \N\mid T_m\leq t\}$. Note that $(Y_t)_t$ is Lipschitz, so
\[
Y_t\leq X_{M(t)}+(t-T_{M(t)})\leq X_{M(t)}+T^*_{M(t)}.
\]
where we recall $T^*_m\coloneqq T_{m+1}-T_m$. We consider these terms separately, using the rough bound that $M(t)\leq t$ for all $t\in \N$. First, we obtain the following by double counting, since $T_m^*=|\{t\mid M(t)=m\}|$:
\[
\sum_{t\leq n}X_{M(t)}\leq\sum_{m\leq M(n)} X_mT^*_m\leq \sum_{m\leq n}X_mT^*_m.
\]
Taking expectations, for any $0<\delta<1/3$ we have
\begin{align}
\sum_{t\leq n}\Ex[X_{M(t)}^\delta]&\leq \sum_{m\leq n}\sum_{k\geq 1}k^\delta\PP[X_m=k]\Ex[T^*_m\mid X_m=k]\label{eq:XMboundedess}\\
&\leq \sum_{m\leq n}\sum_{k\geq 1}k^{\delta}\PP[X_m=k]\cdot Ck\notag\\
&=C\sum_{m\leq n}\Ex[X_m^{1+\delta}],\notag
\end{align}
for some absolute constant $C$ ($C=4$ suffices by \eqref{eq:T*expection} in the proof of \cref{lem:algorithm_step}).
It is similar to bound $T^*_{M(t)}$. Indeed by double counting, we have
\begin{align}\label{eq:TMboundedness}
\sum_{t\leq n}\Ex[(T^*_{M(t)})^\delta]
&\leq \sum_{m\leq n}\Ex[{T_m^*}^{1+\delta}]
=\sum_{m\leq n}\Ex[\Ex[{T_m^*}^{1+\delta}\mid X_m]]
\leq C'\sum_{m\leq n}\Ex[X_m^{1+\delta}],
\end{align}
where the tail bound in~\eqref{eq:algorithmstep3} is used to obtain $\Ex[(T_m^*)^{1+\delta}\mid X_m]\leq C'X_m^{1+\delta}$ for an absolute constant $C'$. Finally, by Lemma \ref{lem:boosted_algorithm}, and using the concave inequality $(x+y)^\delta\leq x^\delta+y^\delta$, we have
\[
\sum_{t\leq n}\Ex[Y_t^\delta]\leq \sum_{t\leq n}\big(\Ex[X_{M(t)}^\delta]+\Ex[(T_{M(t)}^*)^\delta]\big)
\leq (C+C')\sum_{m\leq n}\Ex[X_m^{1+\delta}]\leq C''n,
\]
for some constant $C''$. This completes the proof.
\end{proof}

\subsection{Proofs of technical claims}\label{section:numericalclaims}
\begin{proof}[Proof of Claim~\ref{claim:Ybounds}] Recall that $X\sim \NB(k)$, $(Y\mid X)\sim \NB(\max(X-1,0))$ and $Z\sim \geom(1/2)$, and also $Y,Z$ are independent. First, we have
\[
\Ex[Y]=\Ex[\Ex[Y\mid X]]=\Ex[\max(X-1,0)]=\Ex[X-1+\bbone_{X=0}]=k-1+2^{-k}.
\]
Therefore $\Ex[(Y-Z)\bbone_{Z>0}]=\Ex[Y\bbone_{Z>0}]-\Ex[Z]=\frac 12(k-1+2^{-k})-1$. Next, using the law of total variance,
\begin{align*}
\var[Y]=\Ex[\var[Y\mid X]]+\var[\Ex[Y\mid X]]&=\Ex[2\max(X-1,0)]+\var[\max(X-1,0)]\\
&\leq 2k-2+2^{1-k}+\var[X]\\
&=4k-2+2^{1-k}.
\end{align*}
The memorylessness of the geometric distribution implies that conditioned on $Z > 0$, $Z'=Z-1$ is also a $\geom(1/2)$. Thus,
\begin{align*}
\Ex[(Y-Z-k)^2\bbone_{Z>0}]\cdot 2&=\Ex_{Z'\sim \geom(1/2)}[(Y-Z'-k-1)^2]\\
&\leq \var[Y]+\var[Z']+(\Ex[Y-Z']-k-1)^2\\
&=(4k-2+2^{1-k})+(2)+(2^{-k}-3)^2\\
&< 4k+9,
\end{align*}
giving the desired bounds.
\end{proof}
\begin{proof}[Proof of Claim~\ref{claim:C'bounds}]
Let $k'=k-1$. Recall $M\sim \max(\bin(k',1/2)-1,0)$ and $Z\sim \geom(1/2)$ are independent and $C'$ has conditional distribution $(C'\mid M,Z)\sim \max(Z-\bin(M,1/2),0)$. Define the random variable $M'$ with distribution $(M'\mid M)\sim \bin(M,1/2)$ such that $C'=\max(Z-M',0)$. 

Up to the unpleasantly interleaved $-1$ and max terms, $M' \approx \bin(k', 1/4)$. To justify this approximation, we let $M_1\sim \bin(k',1/4)$ and $M_1'=\max(M_1-1,0)$, and observe that $M'$ stochastically dominates $M_1'$. Let 
\[
C_1'=\max(Z+1-M_1,0)\geq \max(Z-M_1',0),
\]
which stochastically dominates $C'$. Now, we compute $\Ex[C'_1]$ and $\Ex[(C'_1)^2]$ exactly. We have
\begin{align*}
\Ex[C_1']&=\sum_{j=0}^{k'}\PP[M_1=j]\PP[Z+1\geq j+1]\Ex[Z+1-j\mid Z+1\geq j+1]\\
&=\sum_{j=0}^{k'}\PP[M_1=j]2^{-j}(1+\Ex[Z])=2\Ex[2^{-M_1}],
\end{align*}
where the conditional expectation follows from the fact that $Z$ is memoryless. Recall the moment generating function for the binomial distribution $M_1\sim\bin(k', 1/4)$ is $\Ex[e^{tM_1}]=((1-p)+pe^t)^{k'}$, where $p=1/4$. Taking $e^t = \frac12$, we conclude that $\Ex[C'_1]=2(7/8)^{k'}$. Similarly, we compute
\[\Ex[(C'_1)^2]=\Ex[\Ex[\max((Z+1-M_1)^2,0)\mid M_1]]=\Ex[3\cdot 2^{1-M_1}]=6\Ex[2^{-M_1}]=6(7/8)^{k'}.
\]
By stochastic dominance, this gives upper bounds for $\Ex[C']$ and $\Ex[C'^2]$.
\end{proof}

\begin{proof}[Proof of Claim~\ref{claim:secondorderapproxBoostedsequence}]
Define the remainder $R=R(\Delta)\coloneqq (k+\Delta)^p-k^p-pk^{p-1}\Delta-\frac{p(p-1)}{2}k^{p-2}\Delta^2$. Recall that by construction $k+\Delta\geq 0$ deterministically. To bound this remainder, we split into two cases:
\[
\Ex[|R|]=\Ex[|R|\bbone_{|\Delta|<k/2}]+\Ex[|R|\bbone_{|\Delta|\geq k/2}].
\]
Recall that by Lemma~\ref{lem:algorithm_step}, we have $\PP[|\Delta|\geq x]\leq e^{-\Omega(x^2/k(1+x/k))}$. For $x\geq k/2$, the exponent is asymptotically $-\Omega(x)$, and for $x\leq k/2$, the exponent is $-\Omega(x^2/k)$. If $|\Delta|\leq k/2$, Taylor's theorem gives $|R(\Delta)|\leq c_1k^{p-3}|\Delta|^3$ for a constant $c_1$ depending on $p$. Using the tail integral representation, we have
\[
\Ex[|\Delta|^3]=\int_0^{\infty} 3x^2\PP[|\Delta|\geq x]\rm dx\leq \int_0^{\infty}x^2e^{-\Omega(x^2/k)}\rm dx=k^{3/2}\int_0^ku^{2}e^{-\Omega(u^2)}\rm dy=O(k^{3/2}),
\]
where we substitute $u=x/\sqrt k$ so $\rm dx=\sqrt k\rm du$. Therefore $\Ex[|R|\bbone_{|\Delta|< k/2}]\leq O(k^{p-3}k^{3/2})$. For the other term, note that when $|\Delta|\geq k/2$ we have $|R|\leq c_2|\Delta|^{c_3}$ for some constants $c_2,c_3$. Thus we have 
\begin{align*}
\Ex[|R|\bbone_{|\Delta|\geq k/2}]&\leq c_1\Ex[|\Delta|^{c_2}\bbone_{|\Delta|\geq k/2}]\\
&=c_1\int_{0}^\infty c_2x^{c_2-1}\PP\big[|\Delta|\bbone_{|\Delta|\geq k/2}\geq x\big]\rm dx\\
&=c_1\int_{0}^{k/2} c_2x^{c_2-1}\PP[|\Delta|\geq k/2]\rm dx+c_1\int_{k/2}^\infty c_2x^{c_2-1}\PP[|\Delta|\geq x]\rm dx\\
&= O(k^{c_2}e^{-\Omega(k)}) + \int_{k/2}^\infty O(x^{c_2-1})e^{-\Omega(x)}\rm dx\\
&= e^{-\Omega(k)}.
\end{align*}
so combining terms $\Ex[|R|]\leq O(k^{p-3/2})$. Recall from Lemma~\ref{lem:algorithm_step} that $\Ex[\Delta]=-2+O((7/8)^k)$ and $\Ex[\Delta^2]\leq 3k+O(1)$. Then we calculate
\begin{align*}
\Ex[(k-\Delta^p)-k^p]&=\Ex[R]+k^{p-1}\Big(p\Ex[\Delta]+\frac{p(p-1)}{2}k^{-1}\Ex[\Delta^2]\Big)\\
&=\Ex[R]+k^{p-1}\frac p2\Big(-2p+\frac{p(p-1)}{2}3+O(k^{-1})\Big)\\
&=k^{p-1}\Big(\frac p2\big(3p-7\big)+O(k^{-1/2})\Big).
\end{align*}
This gives the desired approximation.
\end{proof}

\section{Concluding Remarks: Larger Alphabets }\label{section:conclusion}

Recall that $\cS_k(n)$ is the set of shuffle squares in $[k]^{2n}$; less is known for alphabet size $k>2$. Observe that $|\cS_k(n)|$ is supermultiplicative in $n$, so the limit $b_k=\lim_{n\to\infty}|\cS_k(n)|^{1/2n}$ exists, and we can write $|\cS_k(n)|=(b_k)^{2n+o(n)}$.  The following upper bound is implied by the proofs of Bukh and Zhou.

\begin{theorem}[{\cite[Theorem 4]{bukhzhou}}]\label{thm:bukhzhou} For $k\geq 1$, $|\cS_k(n)|\leq (\sqrt k+\sqrt{k-1})^{2n+o(n)}<(4k-2-\frac{1}{4k})^{n+o(n)}$.
\end{theorem}

This implies that for fixed $k\geq 4$ and $n\to\infty$, since $\sqrt k+\sqrt{k-1}<k$, an exponentially small fraction of $k$-ary words are shuffle squares. For large fixed alphabets, we conjecture that this bound is essentially tight. We note that the asymptotics of $|\cS_k(n)|$ for fixed $n$, as $k\to\infty$, were resolved by the first author, Huang, Nam, and Thaper~\cite{HHNT}.

\begin{conjecture} As a function of $k\in\N$, we have $\lim_{n\to\infty}|\cS_k(n)|^{1/n}=4k-2-o(1)$.
\end{conjecture}
We give evidence for this conjecture in the following section. Recently, Basu and Ruci\'nski \cite{basu} improved the upper bound for worst-case twins in ternary words, but the average case behavior is wide open.

\begin{problem}\label{prob:ternary} Is it true that $|\cS_3(n)|=(9-o(1))^n$? 
\end{problem}
\noindent This problem is closely related to the extremal problem of determining whether $\rm{LT}(3,n)=(\frac 12-o(1))n$. The following one directional implication follows from work of Bukh and Zhou.
\begin{theorem}[{\cite[Theorem 4]{bukhzhou}}]\label{prop:alphatwins} If $\cS_k(n)\leq (b_k)^{2n+o(n)}$ and there exists $\alpha_k\in (2/k,1)$ such that
\[
\alpha_k\log\Big(\frac{b_k}{\alpha_kk}\Big)+(1-\alpha_k)\log \Big(\frac{k-1}{k(1-\alpha_k)}\Big)<0,
\]
then $\rm{LT}(k,n)< \frac 12\alpha_kn$ for sufficiently large $n$. In particular, if $b_k<k$, then we can take $\alpha_k<1$.
\end{theorem}
\noindent This can be obtained by a union bound over all ``leftmost" embeddings of shuffle squares in a random word $s$. Though we do not offer a conjecture on Problem~\ref{prob:ternary}, numerical evidence suggests that $|\cS_3(n)|=o(9^n)$.

Another problem of combinatorial interest is the behavior of the greedy probabilities $q_t(w)$ as defined in Section \ref{section:greedyalgorithm}, and in particular the limit $c(w)\coloneqq \lim_{t\to\infty} q_{2t+|w|}(w)/q_{2t}(\emptyword)$ for $w\in \Sigma_2$. The existence of this limit is straightforward, but it seems challenging to determine its value. This limit is can be reduced to counting certain $+1,-1$ walks on $\Z$. Lemma~\ref{lem:greedyuniform} provides some monotonicity for $c(w)$ and implies $c(w)\sim \frac 1{2|w|}$ as $|w|\to\infty$. We offer the following example:
\begin{conjecture}We have $c(01)=\pi^2/6-1$.
\end{conjecture}
By solving a relaxed recurrence relation, we know this is within $10^{-8}$ of the truth. It would be interesting to describe $c(w)$ more generally.

\subsection{A Boosted Greedy Algorithm for $k$-ary words}

The greedy algorithm given by~\eqref{eq:greedydefn} generalizes directly to the problem of $k$-ary shuffle squares. A direct enumeration of the `successful' runs of this algorithm gives the lower bound $|\cS_k(n)|\geq \frac 1{n+1}{2n\choose n}(k-1)^n=(4k-4)^{n-o(n)}$ for all $k$. To improve this lower bound, we present a boosted greedy algorithm for $k$-ary words which succeeds exponentially more often.

\begin{proposition}\label{prop:karyboosted} For $k\geq 1$,
\[
\lim_{n\to\infty}|\cS_k(n)|^{1/n}\geq 4k-4+2\log 2+O(1/k)>4k-2.61371+O(1/k).
\]
Moreover,
\[
\lim_{n\to\infty}|\cS_3(n)|^{1/n}\geq 8.133.
\]
\end{proposition}
We omit the proof in this manuscript. The idea is to give a \textit{non-uniform} distribution $\mu$ on input words in $[k]^{2n}$ such that the following boosted algorithm has a reasonable probability of success. Then, one can use the entropy of $\mu$ to show that the algorithm succeeds on a large-cardinality set of input words.
\vskip 5pt
\noindent \textbf{The Algorithm.} We construct quasi-buffers $(W_t)_t\in \big([k]\cup ([k]\times \N)\big)^*$ where pairs $(a,\tau)\in[k]\times \N$ are called `indicators'. We use the recursive definition that $W$ is a \textit{quasi-buffer} for $s[1,t]$ if either $W\in B_t(s)$ or for any partition $W=u_1(a,\tau)u_2(a,\tau)u_3$, both of $u_1u_2au_3$ and $u_1au_2u_3$ are quasi-buffers. In our construction, tuples $(a,t)$ will always appear \textit{in pairs}; these indicate that the letter `$a$' can be in either of \textit{two} positions in the buffer. We also maintain that the first instance of an indicator $(a,\tau)$ is followed by some letter $b\neq a$. We proceed recursively. 

If $W_{t-1}=\emptyword$ or $s(t)\neq W_{t-1}(1)\in[k]$, then as usual $W_t=W_{t-1}\circ s(t)$. If $s(t)=W_{t-1}(1)$, then as usual $W_t=\tail(W_{t-1})$. If $W_{t-1}(1)=(a,\tau)$ for some $\tau\in \N$, we recall $W_{t-1}(2)=b$ for some $b\neq a$. Also there is exactly one other position $j\geq 3$ such that $W_{t-1}(j)=(a,\tau)$; i.e. we can write $W_{t-1}=(a,\tau)bu_2(a,\tau)u_3$ for some words $u_2,u_3$. If $s(t)=a$, then we let $W_{t-1}'=abu_2u_3$, and let $W_t=\tail(W_{t-1}')$ as usual. If $s(t)=b$, then we let $W_{t-1}'=bu_2au_3$, and let $W_t=\tail(W_{t-1}')$. The ``boost" of this algorithm is that when $W_{t-1}(1)$ is an indicator, we are twice as likely to decrease the buffer length. Next, we outline the process of generating indicator pairs. 

Whenever $W_{t-1}$ or $W_{t-1}'=abu_1$ with $a\neq b$ and $s(t)=a$, recall that we set $W_t=bu_1$. Moreover, we record $\tt{last\_match}=a$. Then if the next letters of $s$ are $c_1\cdots c_ra$ with $r>0$ and $c_1,\ldots,c_r\notin \{a,b\}$, then by inspection we have both $W_{t+r+2}=bu_1c_1\cdots c_ra\in B_{t+r+2}$ and also $bu_1ac_1\cdots c_r\in B_{t+r+2}$. We record this information by inserting indicators into the quasi-buffer, redefining $W_{t+r+2}:=bu_1(a,t)c_1\cdots c_r(a,t)$. The indicator generation process succeeds exactly when $s(t+1)=c_1\notin\{a,b\}$, and then the next occurrence in $s$ of $a$ appears before $b$. If the process fails we reset $\tt{last\_match}=\tt{None}$.

\end{document}